\documentclass[12pt]{article}%
\usepackage{hyperref}
\usepackage{amsmath}
\usepackage{amsfonts}
\usepackage{amssymb}
\usepackage{graphicx}%
\newtheorem{theorem}{Theorem}[section]

\newtheorem{corollary}[theorem]{Corollary}

\newtheorem{example}[theorem]{Example}

\newtheorem{lemma}[theorem]{Lemma}

\newtheorem{proposition}[theorem]{Proposition}

\newenvironment{proof}[1][Proof]{\noindent\textbf{#1.} }{\ \rule{0.5em}{0.5em}}
\begin{document}

\title{Evaluating characterizations of homomorphisms on truncated vector lattices of functions}
\author{Karim Boulabiar\thanks{Corresponding author:
\texttt{karim.boulabiar@fst.utm.tn}}\quad and\quad Sameh Bououn\medskip\\{\small Laboratoire de Recherche LATAO}\\{\small D\'{e}partement de Mathematiques, Facult\'{e} des Sciences de Tunis}\\{\small Universit\'{e} de Tunis El Manar, 2092, El Manar, Tunisia}}
\date{\textit{{\small Dedicated to the memory of Professor Abdelmajid Triki}}}
\maketitle

\begin{abstract}
Let $L$ be a (non necessarily unital) truncated vector lattice of real-valued
functions on a nonempty set $X$. A nonzero linear functional $\psi$ on $L$ is
called a truncation homomorphism if it preserves truncation, i.e.,%
\[
\psi\left(  f\wedge\mathbf{1}_{X}\right)  =\min\left\{  \psi\left(  f\right)
,1\right\}  \text{ for all }f\in L.
\]
We prove that a linear functional $\psi$ on $L$ is a truncation homomorphism
if and only if $\psi$ is a lattice homomorphism and%
\[
\sup\left\{  \psi\left(  f\right)  :f\leq\mathbf{1}_{X}\right\}  =1.
\]
This allows us to prove different evaluating characterizations of truncation
homomorphisms. In this regard, a special attention is paid to the continuous
case and various results from the existing literature are generalized.

\end{abstract}



\section{Introduction}

It is well-known that a linear functional $\psi$ on the lattice-ordered
algebra $C\left(  X\right)  $ of all real-valued continuous functions on a
Tychonoff space $X$ is a unital lattice homomorphism if and only if it is an
evaluation at some point $x$ of $X$, i.e.,%
\[
\psi\left(  f\right)  =f\left(  x\right)  \quad\text{for all }f\in C\left(
X\right)
\]
(see, e.g., Theorem 2.33 in \cite{AB06}). In their remarkable papers
\cite{GJ04,GJ01}, Garrido and Jaramillo investigated the extent to which such
a representation can be generalized to a wider class of unital vector
sublattices of $C\left(  X\right)  $. In this regard, they have mainly proved
that if $\psi$ is a linear functional on a unital vector sublattice $L$ of
$C\left(  X\right)  $, then $\psi$ is a unital lattice homomorphism if and
only if $\psi$ is an evaluation at some point in the Stone-\v{C}ech
compactification $\beta X$ of $X$. They obtained, as consequences, some
necessary and sufficient conditions on $L$ for $X$ to be $L$-realcompact,
i.e., any unital lattice homomorphism on $L$ is an evaluation at some point in
$X$. They also used their aforementioned representation theorem to establish
the equivalence between unital lattice homomorphisms and positive algebra
homomorphisms on unital lattice-ordered subalgebras of $C\left(  X\right)  $.
Although they cover a quite large spectrum of function lattices, these
results, relevant as they are, cannot deal with the non-unital case. It seems
to be natural therefore to look beyond the framework of lattices containing
the constant functions. From this point of view, we have thought about vector
sublattices possessing the so-called Stone property. Recall here that a vector
subspace $E$ of the lattice-ordered algebra $\mathbb{R}^{X}$ of all
real-valued functions on an arbitrary non-empty $X$ is said to possess the
\textsl{Stone property} if $E$ contains with any function $f$ the function
$f\wedge\mathbf{1}_{X}$ defined by%
\[
\left(  \mathbf{1}_{X}\wedge f\right)  \left(  x\right)  =1\wedge f\left(
x\right)  =\min\left\{  1,f\left(  x\right)  \right\}  \text{\quad for all
}x\in X,
\]
where $\mathbf{1}_{X}$ is the indicator (or characteristic) function of $X$.
We call a \textsl{truncated vector sublattice} of $\mathbb{R}^{X}$, after
Fremlin in \cite{F74}, any vector sublattice $L$ of $\mathbb{R}^{X}$ which
possesses the Stone property (we do not assume that $\mathbf{1}_{X}$ is
present in $L$). As a matter of fact, the strength of the relationship between
this structure and duality is not a new idea. This goes back to the mid-19th
Century when Stone himself proved that, for every $\sigma$-order continuous
positive linear functional $\psi$ on a truncated vector sublattice $L$ of
$\mathbb{R}^{X}$, there exists a measure $\mu$ on $X$ such that%
\[
\varphi\left(  f\right)  =\int_{X}fd\mu\text{\quad for all }f\in L
\]
(see, e.g., Theorem 4.5.2 in \cite{D04}). This fundamental result is, by now,
referred to as the \textsl{Daniell-Stone Representation Theorem}. It is,
therefore, surprising that there has been no study of evaluating properties of
homomorphisms on truncated vector sublattices of $\mathbb{R}^{X}$. This paper
will actually try to address this omission. Against this background, a
suitable concept of homomorphisms must be introduced, for manifest reasons of
compatibility. To meet this need, we drew inspiration from the recent work
\cite{B14} by Ball to define a \textsl{truncation homomorphism} on the
truncated vector sublattice $L$ of $\mathbb{R}^{X}$ as a \textit{nonzero}
linear functional $\psi$ on $L$ that \textsl{preserves truncation}, i.e.,%
\[
\psi\left(  \mathbf{1}_{X}\wedge f\right)  =1\wedge\psi\left(  \mathbf{1}%
_{X}\right)  =\min\left\{  1,\psi\left(  \mathbf{1}_{X}\right)  \right\}
\text{\quad for all }f\in L.
\]
A short synopsis of the content of this paper seems in order.

In Section 2, the connection between truncation homomorphisms and lattice
homomorphisms are considered in some details. For instance, we prove that any
truncation homomorphism is automatically a lattice homomorphisms, then we find
the missing condition for the converse to hold. The third section contains the
evaluating characterizations of truncation homomorphisms we are looking for.
Indeed, it turns out that a linear functional $\psi$ on a truncated vector
sublattice $L$ of $\mathbb{R}^{X}$ is a truncation homomorphism if and only if
$\psi$ is a $\overline{2}$-\textsl{evaluation} on $L$, i.e., for every $f,g\in
L$ and every $\varepsilon\in\left(  0,\infty\right)  $, the inequalities%
\[
\left\vert f\left(  x\right)  -\psi\left(  f\right)  \right\vert
\leq\varepsilon\quad\text{and\quad}\left\vert g\left(  x\right)  -\psi\left(
g\right)  \right\vert \leq\varepsilon
\]
hold for some $x\in X$ (depending on $f,g$ and $\varepsilon$). Also, we show
that for any truncation homomorphism $\psi$ on a truncated vector sublattice
$L$ of $\mathbb{R}^{X}$ there exists a net $\left(  x_{\sigma}\right)
_{\sigma}$ in $X$ such that%
\[
\lim f\left(  x_{\sigma}\right)  =f\left(  x\right)  \text{ in }%
\mathbb{R}\text{\quad for all }f\in L.
\]
This brings us to the last section, in which the continuous case is
investigated. We prove, among other characterizations, that any truncation
homomorphism on a truncated vector sublattice of $C\left(  X\right)  $ is an
evaluation at some point of $\beta X$ of $X$. As pointed out above, the unital
case was resolved (in an alternative way) by Garrido and Jaramillo. We end the
paper by providing sufficient (and sometimes necessary) conditions on $L$ for
$X$ to be $L$-realcompact, i.e., any truncation homomorphism on $L$ is a
one-point evaluation.

Efforts have been made to make this work more or less accessible to a large
audience in such a way it could be understood by readers with a standard
first-year graduate background on algebra and topology. In spite of that, we
use the great books \cite{AB06,F74} on Vector Lattices, \cite{GJ76,KM97} on
Real-Valued Functions, and \cite{E89,W70} on General Topology as sources for
unexplained terminology and notation (unless otherwise stated explicitly).

\section{Connection with lattice homomorphisms}

Our first discussion may well not have been quite on the agenda, but we think
that it is sufficiently interesting to be incorporated into the text. Recall
from the introduction that a vector subspace $E$ of $\mathbb{R}^{X}$ is said
to possess the \textsl{Stone property} if%
\[
f\wedge\mathbf{1}_{X}\in E\text{\quad for all }f\in E.
\]
It turns out that any vector subspace $\mathbb{R}^{X}$ possessing the Stone
property is a vector sublattice of $\mathbb{R}^{X}$, provided that it contains
the constant functions.

\begin{proposition}
Let $E$ be a vector subspace of $\mathbb{R}^{X}$ such that $\mathbf{1}_{X}\in
E$. Then $E$ possesses the Stone property if and only if $E$ is vector
sublattice of $\mathbb{R}^{X}$.
\end{proposition}

\begin{proof}
Sufficiency being straightforward, we prove Necessity. Assume that $E$
possesses the Stone property and choose $f\in E$. Observe that%
\[
\left\vert \mathbf{1}_{X}-f\right\vert =\mathbf{1}_{X}+f-2\left(
f\wedge\mathbf{1}_{X}\right)  \in E
\]
Thus, if $g$ is an arbitrary element in $E$ then $\mathbf{1}_{X}-g\in E$ and
so%
\[
\left\vert g\right\vert =\left\vert \mathbf{1}_{X}-\left(  \mathbf{1}%
_{X}-g\right)  \right\vert \in E.
\]
This implies that $E$ is a vector sublattice of $\mathbb{R}^{X}$, as required.
\end{proof}

We thought at a moment that the result should hold for any vector subspace of
$\mathbb{R}^{X}$. However, both implications are not true in general as the
following examples show.

\begin{example}
\begin{enumerate}
\item[\emph{(i)}] A function $f\in\mathbb{R}^{\left[  0,\infty\right)  }$ is
said to be \textsl{essentially linear} if there exists $x_{f}>0$ and $a_{f}%
\in\mathbb{R}$ such that%
\[
f\left(  x\right)  =a_{f}x\quad\text{for all }x\in\left(  x_{f},\infty\right)
.
\]
It is not hard to see that the subset $L$ of $\mathbb{R}^{\left[
0,\infty\right)  }$ is a vector sublattice of $\mathbb{R}^{\left[
0,\infty\right)  }$. However, if%
\[
i\left(  x\right)  =x\quad\text{for all }x\in\left[  0,\infty\right)
\]
then $i\in L$, while $i\wedge\mathbf{1}_{\left[  0,\infty\right)  }\notin L$.
This means that $L$ does not possess the Stone property.

\item[\emph{(ii)}] For every $f\in C\left(  \mathbb{R}\right)  $, we define
$\tau f\in C\left(  \mathbb{R}\right)  $ by%
\[
\left(  \tau f\right)  \left(  x\right)  =xf\left(  x\right)  \quad\text{for
all }x\in\mathbb{R}.
\]
Obviously, the set%
\[
E=\left\{  \tau f:f\in C\left(  \mathbb{R}\right)  \right\}
\]
is a vector subspace of $\mathbb{R}^{\mathbb{R}}$. Choose $f\in E$ and define
$g\in C\left(  \mathbb{R}\right)  $ by%
\[
g\left(  x\right)  =\frac{f\left(  x\right)  }{\sup\left\{  f\left(  x\right)
,1\right\}  }\quad\text{for all }x\in\mathbb{R}.
\]
It is an elementary exercise to show that%
\[
\tau f\wedge\mathbf{1}_{\mathbb{R}}=\tau g\in E,
\]
meaning that $E$ does possess the Stone property. Nevertheless, $E$ is not a
vector sublattice of $\mathbb{R}^{\mathbb{R}}$ since, if $e$ is the function
given by%
\[
e\left(  x\right)  =x\quad\text{for all }x\in\mathbb{R},
\]
then $e\in E$ while $\left\vert e\right\vert \notin E$.
\end{enumerate}
\end{example}

Now, let's get to the heart of the matter. As before, we call a
\textsl{truncated vector sublattice} of $\mathbb{R}^{X}$ any vector sublattice
$L$ of $\mathbb{R}^{X}$ possessing the Stone property. We emphasize that we do
not assume that truncated vector sublattices of $\mathbb{R}^{X}$ contain
$\mathbf{1}_{X}$. A vector sublattice of $\mathbb{R}^{X}$ containing
$\mathbf{1}_{X}$ is called a \textsl{unital vector sublattice} of
$\mathbb{R}^{X}$. Obviously, any unital vector sublattice of $\mathbb{R}^{X}$
is a truncated vector sublattice of $\mathbb{R}^{X}$. A \textit{nonzero}
linear functional $\psi$ on the truncated vector sublattice $L$ of
$\mathbb{R}^{X}$ is called a \textsl{truncation homomorphism} if%
\[
\psi\left(  f\wedge\mathbf{1}_{X}\right)  =\psi\left(  f\right)  \wedge
1=\min\left\{  \psi\left(  f\right)  ,1\right\}  \quad\text{for all }f\in L.
\]
Also, recall that the linear functional $\psi$ on a vector sublattice $L$ of
$\mathbb{R}^{X}$ is called a \textsl{lattice homomorphism} if%
\[
\psi\left(  f\wedge g\right)  =\psi\left(  f\right)  \wedge\psi\left(
g\right)  =\min\left\{  \psi\left(  f\right)  ,\psi\left(  g\right)  \right\}
\quad\text{for all }f,g\in L.
\]
Clearly, a linear functional $\psi$ on $L$ is a lattice homomorphism if and
only if%
\[
\left\vert \psi\left(  f\right)  \right\vert =\psi\left(  \left\vert
f\right\vert \right)  \quad\text{for all }f\in L.
\]
Notice that any lattice homomorphism $\psi$ on the vector sublattice $L$ of
$\mathbb{R}^{X}$ is \textsl{positive} (and thus increasing), that is to say,%
\[
\psi\left(  f\right)  \geq0\quad\text{for all }f\in L^{+},
\]
where $L^{+}$ denotes the set of all positive functions in $L$.

Connections between truncation homomorphisms and lattice homomorphisms on
truncated vector sublattices of functions are studied next.

\begin{lemma}
\label{Stone-Riesz}Any truncation homomorphism on a truncated vector
sublattice $L$ of $\mathbb{R}^{X}$ is a lattice homomorphism on $L$.
\end{lemma}

\begin{proof}
Let $\psi$ be a truncation homomorphism on the truncated vector sublattice $L$
of $\mathbb{R}^{X}$. First, we claim that $\psi$ is positive. To this end,
choose $f\in L$ and $n\in\left\{  1,2,...\right\}  $. If $f\leq0$ then%
\[
n\psi\left(  f\right)  =\psi\left(  nf\right)  =\psi\left(  \mathbf{1}%
_{X}\wedge nf\right)  =1\wedge\psi\left(  nf\right)  \leq1.
\]
It follows that $\psi\left(  f\right)  \leq0$ because $n$ is arbitrary in
$\left\{  1,2,...\right\}  $. This yields that $\psi$ is positive, as
required. Now, let $f\in L$ and observe that%
\[
0\leq n\psi\left(  f^{+}\right)  -n\psi\left(  f\right)  ^{+}%
\]
because $\psi$ is positive. Moreover,%
\begin{align*}
n\psi\left(  f^{+}\right)   &  \leq\psi\left(  nf^{+}\right)  -1\wedge
\psi\left(  nf^{+}\right)  +1=\psi\left(  nf^{+}\right)  -\psi\left(
\mathbf{1}_{X}\wedge nf^{+}\right)  +1\\
&  =\psi\left(  nf^{+}-\mathbf{1}_{X}\wedge nf^{+}\right)  +1=\psi\left(
nf-\mathbf{1}_{X}\wedge nf\right)  +1\\
&  =\psi\left(  nf\right)  -1\wedge\psi\left(  nf\right)  +1=\psi\left(
nf\right)  ^{+}-1\wedge\psi\left(  nf\right)  ^{+}+1\\
&  \leq\psi\left(  nf\right)  ^{+}+1.
\end{align*}
Therefore,%
\[
0\leq n\left[  \psi\left(  f^{+}\right)  -\psi\left(  f\right)  ^{+}\right]
\leq1.
\]
We derive that $\psi\left(  f^{+}\right)  =\psi\left(  f\right)  ^{+}$ and the
proof is complete.
\end{proof}

It is all too clear that the converse of Lemma \ref{Stone-Riesz} fails. It is
natural therefore to ask for the missing condition for a lattice homomorphism
on a truncated vector sublattice of $\mathbb{R}^{X}$ to be a truncation
homomorphism. The following theorem answers this question.

\begin{theorem}
\label{equivalence}Let $\psi$ be a linear functional on a truncated vector
sublattice $L$ of $\mathbb{R}^{X}$. Then the following are equivalent.

\begin{enumerate}
\item[\emph{(i)}] $\psi$ is a truncation homomorphism.

\item[\emph{(ii)}] $\psi$ is a lattice homomorphism and%
\[
\sup\left\{  \psi\left(  f\right)  :f\in L\text{ and }f\leq\mathbf{1}%
_{X}\right\}  =1.
\]

\end{enumerate}
\end{theorem}

\begin{proof}
$\mathrm{(i)\Rightarrow(ii)}$ Assume that $\psi$ is a truncation homomorphism.
By Lemma \ref{Stone-Riesz}, $\psi$ is a lattice homomorphism. On the other
hand, if $f\in L$ with $f\leq\mathbf{1}_{X}$ then $f=f\wedge\mathbf{1}_{X}$
and so%
\[
\psi\left(  f\right)  =\psi\left(  f\wedge\mathbf{1}_{X}\right)  =\psi\left(
f\right)  \wedge1\leq1.
\]
Moreover, let $\lambda\in\mathbb{R}$ and suppose that $\lambda\geq\psi\left(
f\right)  $ for every $f\in L$ with $f\leq\mathbf{1}_{X}$. Choose $f\in L$
such that $\psi\left(  f\right)  \neq0$ (such a function $f$ exists because,
by definition, $\psi\neq0$) and put%
\[
g=\mathbf{1}_{X}\wedge\frac{1}{\left\vert \psi\left(  f\right)  \right\vert
}\left\vert f\right\vert .
\]
Clearly, $g\in L$ and $g\leq\mathbf{1}_{X}$. Hence,%
\[
\lambda\geq\psi\left(  g\right)  =\psi\left(  \mathbf{1}_{X}\wedge\frac
{1}{\left\vert \psi\left(  f\right)  \right\vert }\left\vert f\right\vert
\right)  =1\wedge\frac{\psi\left(  \left\vert f\right\vert \right)
}{\left\vert \psi\left(  f\right)  \right\vert }=1\wedge\frac{\left\vert
\psi\left(  f\right)  \right\vert }{\left\vert \psi\left(  f\right)
\right\vert }=1
\]
(where we use once again Lemma \ref{Stone-Riesz}). This means that%
\[
1=\sup\left\{  \psi\left(  f\right)  :f\in L\text{ and }f\leq\mathbf{1}%
_{X}\right\}  ,
\]
as desired.

$\mathrm{(ii)\Rightarrow(i)}$ Let $f\in L$ and observe that $\psi\left(
f\wedge\mathbf{1}_{X}\right)  \leq\psi\left(  f\right)  $ because $\psi$ is
positive. Moreover, $f\wedge\mathbf{1}_{X}\in L$ and $f\wedge\mathbf{1}%
_{X}\leq\mathbf{1}_{X}$, so, $\psi\left(  f\wedge\mathbf{1}_{X}\right)  \leq
1$. This means that $\psi\left(  f\wedge\mathbf{1}_{X}\right)  \leq\psi\left(
f\right)  \wedge1$. Conversely, pick $g\in L$ with $g\leq\mathbf{1}_{X}$ and
observe that from $f\wedge g\leq f\wedge\mathbf{1}_{X}$ it follows that%
\[
\psi\left(  f\wedge\mathbf{1}_{X}\right)  \geq\psi\left(  f\wedge g\right)
=\psi\left(  f\right)  \wedge\psi\left(  g\right)  .
\]
Accordingly,%
\begin{align*}
\psi\left(  f\wedge\mathbf{1}_{X}\right)   &  \geq\sup\left\{  \psi\left(
f\right)  \wedge\psi\left(  g\right)  :g\in L\text{ and }g\leq\mathbf{1}%
_{X}\right\} \\
&  =\psi\left(  f\right)  \wedge\sup\left\{  \psi\left(  g\right)  :g\in
L\text{ and }g\leq\mathbf{1}_{X}\right\} \\
&  =\psi\left(  f\right)  \wedge1.
\end{align*}
This ends the proof of the theorem.
\end{proof}

A lattice homomorphism $\psi$ on a unital vector sublattice $L$ of
$\mathbb{R}$ is said to be \textsl{unital} if $\psi\left(  \mathbf{1}%
_{X}\right)  =1$. Hence, we get the following as a direct inference of Theorem
\ref{equivalence}.

\begin{corollary}
A linear functional $\psi$ on a unital vector sublattice of $\mathbb{R}^{X}$
is a truncation homomorphism if and only if $\psi$ is a unital lattice homomorphism.
\end{corollary}

\section{Evaluating characterizations}

We start this section with two technical lemmas.

\begin{lemma}
\label{Tech_1}If $f,g\in\mathbb{R}^{X}$ then%
\[
\left\vert f+\mathbf{1}_{X}\right\vert \wedge g=\left\vert f\right\vert
-2\left(  f^{-}\wedge\mathbf{1}_{X}\right)  +\left(  g-\left\vert f\right\vert
+2\left(  f^{-}\wedge\mathbf{1}_{X}\right)  \right)  \wedge\mathbf{1}_{X}.
\]

\end{lemma}

\begin{proof}
First, observe that%
\begin{align*}
\left\vert f+\mathbf{1}_{X}\right\vert  &  =\left(  f+\mathbf{1}_{X}\right)
\vee\left(  -f-\mathbf{1}_{X}\right) \\
&  =\left[  \left(  f+\mathbf{1}_{X}\right)  \vee\left(  f-\mathbf{1}%
_{X}\right)  \right]  \vee\left(  -f-\mathbf{1}_{X}\right) \\
&  =\left(  f+\mathbf{1}_{X}\right)  \vee\left[  \left(  f-\mathbf{1}%
_{X}\right)  \vee\left(  -f-\mathbf{1}_{X}\right)  \right] \\
&  =\left(  \left\vert f\right\vert -2f^{-}+\mathbf{1}_{X}\right)  \vee\left(
\left\vert f\right\vert -\mathbf{1}_{X}\right) \\
&  =\left\vert f\right\vert -\left[  \left(  2f^{-}-\mathbf{1}_{X}\right)
\wedge\mathbf{1}_{X}\right] \\
&  =\left\vert f\right\vert -2\left(  f^{-}\wedge\mathbf{1}_{X}\right)
+\mathbf{1}_{X}.
\end{align*}
It follows that%
\begin{align*}
\left\vert f+\mathbf{1}_{X}\right\vert \wedge g  &  =\left(  \left\vert
f\right\vert -2\left(  f^{-}\wedge\mathbf{1}_{X}\right)  +\mathbf{1}%
_{X}\right)  \wedge g\\
&  =\left\vert f\right\vert -2\left(  f^{-}\wedge\mathbf{1}_{X}\right)
+\left(  g-\left\vert f\right\vert +2\left(  f^{-}\wedge\mathbf{1}_{X}\right)
\right)  \wedge\mathbf{1}_{X},
\end{align*}
which is the desired equality.
\end{proof}

\begin{lemma}
\label{Tech_2}Let $L$ be a truncated vector sublattice of $\mathbb{R}^{X}$ and
$\psi$ be truncation homomorphism on $L$. Then%
\[
\left\vert f+\lambda\right\vert \wedge g\in L\quad\text{and\quad}\psi\left(
\left\vert f+\lambda\right\vert \wedge g\right)  =\left\vert \psi\left(
f\right)  +\lambda\right\vert \wedge\psi\left(  g\right)
\]
hold for all $f,g\in L$ and $\lambda\in\mathbb{R}.$
\end{lemma}

\begin{proof}
Let $f,g\in L$ and $\lambda\in\mathbb{R}$. In view of Theorem
\ref{Stone-Riesz}, we have nothing to prove if $\lambda=0$. So, assume that
$\lambda>0$. For the sake of brevity, we put%
\[
u=\frac{1}{\lambda}f\quad\text{and\quad}v=\frac{1}{\lambda}g.
\]
Notice that $u,v\in L$. Using Lemma \ref{Tech_1}, we get%
\[
\left\vert f+\lambda\right\vert \wedge g=\lambda\left(  \left\vert
u+\mathbf{1}_{X}\right\vert \wedge v\right)  \in L.
\]
Furthermore, Theorem \ref{Stone-Riesz} together with Lemma \ref{Tech_1} yields
that%
\begin{align*}
\psi\left(  \left\vert u+\mathbf{1}_{X}\right\vert \wedge v\right)   &
=\psi\left(  \left\vert u\right\vert -2\left(  u^{-}\wedge\mathbf{1}%
_{X}\right)  +\left(  v-\left\vert u\right\vert +2\left(  u^{-}\wedge
\mathbf{1}_{X}\right)  \right)  \wedge\mathbf{1}_{X}\right) \\
&  =\left\vert \psi\left(  u\right)  \right\vert -2\left(  \psi\left(
u\right)  ^{-}\wedge1\right)  +1\wedge\left(  \psi\left(  v\right)
-\left\vert \psi\left(  u\right)  \right\vert +2\left(  \psi\left(  u\right)
^{-}\wedge1\right)  \right) \\
&  =\left(  \left\vert \psi\left(  u\right)  +1\right\vert \wedge\psi\left(
v\right)  \right)  =\frac{1}{\lambda}\left(  \left\vert \psi\left(  f\right)
+\lambda\right\vert \wedge\psi\left(  g\right)  \right)  .
\end{align*}
Thus,%
\[
\psi\left(  \left\vert f+\lambda\right\vert \wedge g\right)  =\left\vert
\psi\left(  f\right)  +\lambda\right\vert \wedge\psi\left(  g\right)  .
\]
Now, suppose that $\lambda<0$. By the positive case, we have%
\[
\left\vert f+\lambda\right\vert \wedge g=\left\vert -f-\lambda\right\vert
\wedge g\in L
\]
and, analogously,%
\begin{align*}
\psi\left(  \left\vert f+\lambda\right\vert \wedge g\right)   &  =\psi\left(
\left\vert -f-\lambda\right\vert \wedge g\right) \\
&  =\left\vert \psi\left(  -f\right)  -\lambda\right\vert \wedge\psi\left(
g\right) \\
&  =\left\vert \psi\left(  f\right)  +\lambda\right\vert \wedge\psi\left(
g\right)  .
\end{align*}
This completes the proof of the lemma.
\end{proof}

At this point, let $n\in\left\{  1,2,...\right\}  $. A linear functional
$\psi$ on a vector subspace $L$ of $\mathbb{R}^{X}$ is called an $\overline
{n}$\textsl{-evaluation} if, for every $f_{1},...,f_{n}\in L$ and every
$\varepsilon\in\left(  0,\infty\right)  $, there exists $x\in X$ such that%
\[
\left\vert \psi\left(  f_{k}\right)  -f_{k}\left(  x\right)  \right\vert
\leq\varepsilon\quad\text{for all }k\in\left\{  1,...,n\right\}  .
\]
We have gathered now all the ingredients we need to prove the central theorem
of this section.

\begin{theorem}
\label{evaluation}Let $L$ be a truncated vector sublattice of $\mathbb{R}^{X}$
and $\psi$ be a linear functional on $L$. Then the following are equivalent.

\begin{enumerate}
\item[\emph{(i)}] $\psi$ is a truncation homomorphism on $L$.

\item[\emph{(ii)}] $\psi$ is an $\overline{n}$-evaluation on $L$ for all
$n\in\left\{  1,2,...\right\}  $.

\item[\emph{(iii)}] $\psi$ is an $\overline{2}$-evaluation on $L$.

\item[\emph{(iv)}] There exists a net $\left(  x_{\sigma}\right)  $ of
elements of $X$ such that%
\[
\psi\left(  f\right)  =\lim f\left(  x_{\sigma}\right)  \text{ in }%
\mathbb{R}\text{\quad for all }f\in L.
\]

\end{enumerate}
\end{theorem}

\begin{proof}
First, observe that the implications $\mathrm{(ii)\Rightarrow(iii)}$ and
$\mathrm{(iv)\Rightarrow(i)}$ are obvious. The other implications are quite involved.

$\mathrm{(i)\Rightarrow(ii)}$ Let $n\in\left\{  1,2,...\right\}  $ and choose
$e\in L$ such that $\psi\left(  e\right)  \neq0$. Using Theorem
\ref{Stone-Riesz}, we get%
\[
\psi\left(  \left\vert e\right\vert \right)  =\left\vert \psi\left(  e\right)
\right\vert >0.
\]
So, by replacing $e$ by $\left\vert e\right\vert \wedge\mathbf{1}_{X}$ (if
needed), we can assume that $0\leq e\leq\mathbf{1}_{X}$ in $\mathbb{R}^{X}$.
Let $\varepsilon\in\left(  0,\infty\right)  $ and put $\theta=\min\left\{
1,\varepsilon\right\}  $. Given $f_{1},f_{2},...,f_{n}\in L$, we define%
\[
h=\frac{1}{\theta}\left(  e\wedge%
{\displaystyle\bigvee_{k=1}^{n}}
\left\vert f_{k}-\psi\left(  f_{k}\right)  \right\vert \right)  \in
\mathbb{R}^{X}.
\]
From Lemma \ref{Tech_2} it follows that $h\in L$. Moreover, as easy
calculation based on Theorem \ref{Stone-Riesz} and Lemma \ref{Tech_2} yields
that $\psi\left(  h\right)  =0$. It follows that%
\[
\psi\left(  h-e\right)  =-\psi\left(  e\right)  <0.
\]
Therefore, there exists $x\in X$ such that $h\left(  x\right)  -e\left(
x\right)  <0$ (because $\psi$ is positive). We derive that%
\[
e\left(  x\right)  \wedge%
{\displaystyle\bigvee_{k=1}^{n}}
\left\vert f_{k}\left(  x\right)  -\psi\left(  f_{k}\right)  \right\vert
=\theta h\left(  x\right)  <\theta e\left(  x\right)  .
\]
Since $\theta\in\left(  0,1\right]  $ and $0<e\left(  x\right)  \leq1$, we
obtain%
\[%
{\displaystyle\bigvee_{k=1}^{n}}
\left\vert f_{k}\left(  x\right)  -\psi\left(  f_{k}\right)  \right\vert
\leq\theta e\left(  x\right)  \leq\theta\leq\varepsilon
\]
and $\mathrm{(ii)}$ follows.

$\mathrm{(iii)\Rightarrow(i)}$ Let $f\in L$ and observe that if $\varepsilon
\in\left(  0,\infty\right)  $ then there exists $x\in X$ such that%
\[
\left\vert \left(  \mathbf{1}_{X}\wedge f\right)  \left(  x\right)
-\psi\left(  \mathbf{1}_{X}\wedge f\right)  \right\vert \leq\frac{\varepsilon
}{2}\quad\text{and\quad}\left\vert f\left(  x\right)  -\psi\left(  f\right)
\right\vert \leq\frac{\varepsilon}{2}.
\]
By the classical Birkhoff's Inequality (see, e.g., Theorem 1.9 (ii) in
\cite{AB06}), we derive that%
\begin{align*}
\left\vert 1\wedge\psi\left(  f\right)  -\psi\left(  \mathbf{1}_{X}\wedge
f\right)  \right\vert  &  =\left\vert \left(  \mathbf{1}_{X}\wedge f\right)
\left(  x\right)  -1\wedge\psi\left(  f\right)  \right\vert \\
&  +\left\vert \left(  \mathbf{1}_{X}\wedge f\right)  \left(  x\right)
-\psi\left(  \mathbf{1}_{X}\wedge f\right)  \right\vert \\
&  \leq1\wedge\left\vert f\left(  x\right)  -\psi\left(  f\right)  \right\vert
+\frac{\varepsilon}{2}\leq1\wedge\frac{\varepsilon}{2}+\frac{\varepsilon}%
{2}\leq\varepsilon.
\end{align*}
Since $\varepsilon$ is arbitrary in $\left(  0,\infty\right)  $, we conclude
that $\psi\left(  \mathbf{1}_{X}\wedge f\right)  =1\wedge\psi\left(  f\right)
$. This means that $\psi$ is a truncation homomorphism on $L$, as required.

$\mathrm{(ii)\Rightarrow(iv)}$ First, assume that $L$ separates the points of
$X$ and define a map $J:X\rightarrow\mathbb{R}^{L}$ by%
\[
J\left(  x\right)  =\left(  f\left(  x\right)  \right)  _{f\in L}\text{ for
all }x\in X.
\]
By the separation condition, the map $J$ is one-to-one and thus $X$ can be
considered as a subset of $\mathbb{R}^{L}$. Let $\mathbb{R}^{L}$ be endowed
with its usual Tychonoff product topology and put $\omega=\left(  \psi\left(
f\right)  \right)  _{f\in L}$. Denote by $\Omega$ a neighborhood of $\omega$
in $\mathbb{R}^{L}$. There exists $\varepsilon\in\left(  0,\infty\right)  $
and a non-empty finite subset $A$ of $L$ such that%
\[
\omega\in\prod_{f\in L}\Omega_{f}\subset\Omega,
\]
where%
\[
\Omega_{f}=\mathbb{R}\text{ if }f\notin A\quad\text{and\quad}\Omega
_{f}=\left(  \psi\left(  f\right)  -\alpha,\psi\left(  f\right)
+\varepsilon\right)  \text{ if }f\in A.
\]
By Theorem \ref{evaluation}, there exists $x\in X$ such that%
\[
\left\vert \psi\left(  f\right)  -f\left(  x\right)  \right\vert
<\varepsilon\text{ for all }f\in A.
\]
This yields that $x\in\Omega\cap X$ and that $\omega\in\overline{X}$, where
$\overline{X}$ is the closure of $X$ in $\mathbb{R}^{L}$. It follows that
there exists a net $\left(  x_{\sigma}\right)  $ in $X$ converging to $\omega
$, i.e.,%
\[
\lim x_{\sigma}=\left(  \psi\left(  f\right)  \right)  _{f\in L}\text{ in
}\mathbb{R}^{L}.
\]
Choose $f\in L$ and use the continuity of the projection $\pi_{f}%
:\mathbb{R}^{L}\rightarrow\mathbb{R}$ defined by%
\[
\pi_{f}\left(  \psi\right)  =\psi\left(  f\right)  \text{ for all }\psi
\in\mathbb{R}^{L}%
\]
to write%
\[
\lim f\left(  x_{\sigma}\right)  =\lim\pi_{f}\left(  x_{\sigma}\right)
=\pi_{f}\left(  \lim x_{\sigma}\right)  =\psi\left(  f\right)  .
\]
Now, we discuss the general case. An equivalence relation $\sim$ can be
defined on $X$ by putting $x\sim y$ if and only if $f\left(  x\right)
=f\left(  y\right)  $ for all $f\in L$. The set of all equivalence classes
$\widetilde{x}$ is denoted by $\widetilde{X}$. Define a map $T$ from $L$ into
$\mathbb{R}^{\widetilde{X}}$ by putting%
\[
T\left(  f\right)  \left(  \widetilde{x}\right)  =f\left(  x\right)  \text{
for all }f\in L\text{ and }x\in X.
\]
Clearly, $T$ is well-defined and it is linear. Moreover, if $f\in L$ and $x\in
X$ then%
\[
T\left(  f\wedge\mathbf{1}_{X}\right)  \left(  \widetilde{x}\right)  =\left(
f\wedge\mathbf{1}_{X}\right)  \left(  x\right)  =f\left(  x\right)
\wedge1=T\left(  f\right)  \left(  \widetilde{x}\right)  \wedge1=\left(
T\left(  f\right)  \wedge\mathbf{1}_{X^{\sim}}\right)  \left(  \widetilde
{x}\right)  .
\]
Define a map $\widetilde{\psi}$ from $T\left(  L\right)  $ to $\mathbb{R}$ by%
\[
\widetilde{\psi}\left(  T\left(  f\right)  \right)  =\psi\left(  f\right)
\text{ for all }f\in L.
\]
This map is obviously well-defined and it is a truncation homomorphism on
$T\left(  L\right)  $, which is a truncated vector sublattice of
$\mathbb{R}^{\widetilde{X}}$. Since $T\left(  L\right)  $ separates the points
of $\widetilde{X}$, the first case guaranties the existence of a net $\left(
x_{\sigma}\right)  $ in $X$ such that%
\[
\psi\left(  f\right)  =\widetilde{\psi}\left(  T\left(  f\right)  \right)
=\lim\left(  Tf\right)  \left(  \widetilde{x_{\sigma}}\right)  =\lim f\left(
x_{\sigma}\right)  .
\]
This completes the proof of theorem.
\end{proof}

It should be pointed out that Theorem \ref{evaluation} provides the optimal
evaluating characterization of truncation homomorphisms. Indeed, consider the
linear form $\psi$ defined on the unital vector sublattice $C\left(  \left[
0,1\right]  \right)  $ of $\mathbb{R}^{\left[  0,1\right]  }$ by%
\[
\psi\left(  f\right)  =\int_{0}^{1}f\left(  x\right)  dx\quad\text{for all
}f\in C\left(  \left[  0,1\right]  \right)  .
\]
By the first Mean Value Theorem for Definite Integrals, we see that $\psi$ is
a $\overline{1}$-evaluation. However, $\varphi$ is far from being a truncation
homomorphism since it is not a lattice homomorphism.

Recall at this point that $\mathbb{R}^{X}$ is also an associative algebra with
respect to the pointwise product. Also, recall that a linear functional $\psi$
on a subalgebra $A$ of $\mathbb{R}^{X}$ is called an \textsl{algebra
homomorphism} if%
\[
\psi\left(  fg\right)  =\psi\left(  f\right)  \psi\left(  g\right)  \text{ for
all }f,g\in A.
\]
The last result of this section extends the equivalence
$\mathrm{(i)\Leftrightarrow(ii)}$ of \cite[Lemma 2.3]{GJ04} in two directions
(the $C\left(  X\right)  $-case is well-known and can be found, for instance,
in \cite{B15} or \cite{EO07}). On the one hand, the subalgebra under
consideration is not assumed to contains $\mathbf{1}_{X}$ and, on the other
hand, functions in this subalgebra need not be continuous (no topology is involved).

\begin{corollary}
Let $A$ be a truncated vector sublattice and a subalgebra of $\mathbb{R}^{X}$.
A linear functional $\psi$ on $A$ is a Stone homomorphism if and only if
$\psi$ is a positive algebra homomorphism.
\end{corollary}

\begin{proof}
The `only if' part follows immediately from the implication
$\mathrm{(i)\Rightarrow(iv)}$ in Theorem \ref{evaluation}. Conversely, suppose
that $\psi$ is a positive algebra homomorphism. We claim that $\psi$ is a
truncation homomorphism. To this end, we shall use Theorem \ref{equivalence}.
First, pick $f\in A$ and observe that%
\[
\psi\left(  \left\vert f\right\vert \right)  ^{2}=\psi\left(  \left\vert
f\right\vert ^{2}\right)  =\psi\left(  f^{2}\right)  =\psi\left(  f\right)
^{2}.
\]
Since $\psi$ is positive, $\psi\left(  \left\vert f\right\vert \right)  \geq0$
and thus $\psi\left(  \left\vert f\right\vert \right)  =\left\vert \psi\left(
f\right)  \right\vert $. We conclude that $\psi$ is a lattice homomorphism.
Now, let $f\in A$ such that $f\leq0$. Since $\psi$ is positive, $\psi\left(
f\right)  \leq0$. Moreover, if $0\leq f\leq1$ then $f^{2}\leq f$ from which it
follows that%
\[
0\leq\psi\left(  f\right)  ^{2}\leq\psi\left(  f^{2}\right)  \leq\psi\left(
f\right)  .
\]
Hence, either $\psi\left(  f\right)  =0$ or $\psi\left(  f\right)  \leq1$. In
summary, if $f\in A$ with $f\leq1$ then $\psi\left(  f\right)  \leq1$. We
derive that the supremum%
\[
a=\sup\left\{  \psi\left(  f\right)  :f\in A\text{ and }f\leq1\right\}
\]
exists in $\mathbb{R}^{+}$. Clearly,%
\[
a=\sup\left\{  \psi\left(  f\right)  :f\in A\text{ and }0\leq f\leq1\right\}
\]
and so%
\begin{align*}
a^{2} &  =\sup\left\{  \psi\left(  f\right)  ^{2}:f\in A\text{ and }0\leq
f\leq1\right\}  \\
&  =\sup\left\{  \psi\left(  f^{2}\right)  :f\in A\text{ and }0\leq
f\leq1\right\}  \\
&  =\sup\left\{  \psi\left(  f\right)  :f\in A\text{ and }0\leq f\leq
1\right\}  =a.
\end{align*}
We derive that $a=0$ or $a=1$. If $a=0$ then, obviously, $\psi=0$ which is not
the case. Thus $a=1$ and the corollary follows.
\end{proof}

\section{Continuous case}

In order to avoid unnecessary repetition we will assume throughout this
section that $X$ is a Tychonoff space. Any truncated vector sublattice of
$\mathbb{R}^{X}$ which is contained in $C\left(  X\right)  $ is called a
\textsl{truncated vector sublattice} of $C\left(  X\right)  $. In the first
result of this section, we shall prove that any truncation homomorphism on a
truncated vector sublattice of $C\left(  X\right)  $ is an evaluation at some
point of the Stone-\v{C}ech compactification $\beta X$ of $X$. The unital
version of this representation theorem has been obtained with a completely
different approach used by Garrido and Jaramillo in \cite{GJ04,GJ01}. Still,
we need to recall that any $f\in C\left(  X\right)  $ can be extended uniquely
to a continuous function $f^{\beta}$ from $\beta X$ into the one-point
compactification $\mathbb{R}\cup\left\{  \infty\right\}  $ of $\mathbb{R}$.

\begin{theorem}
\label{Cech}Let $L$ be a truncated vector sublattice of $C\left(  X\right)  $.
A nonzero linear functional $\psi$ on $L$ is a truncation homomorphism if and
only if there exists $u\in\beta X$ such that%
\[
\psi\left(  f\right)  =f^{\beta}\left(  u\right)  \quad\text{for all }f\in L.
\]

\end{theorem}

\begin{proof}
The `if' part being obvious, we prove the `only if' part. Assume that $\psi$
is a truncation homomorphism. By Theorem \ref{evaluation}, there exists a net
$\left(  x_{\sigma}\right)  $ in $X$ such that%
\[
\lim f\left(  x_{\sigma}\right)  =\psi\left(  f\right)  \text{ in }%
\mathbb{R}\quad\text{for all }f\in L.
\]
Replacing if necessary $\left(  x_{\sigma}\right)  $ by a subnet, we may
suppose that $\left(  x_{\sigma}\right)  $ converges to some $u\in\beta X$.
Take $f\in L$ and observe that%
\[
f^{\beta}\left(  u\right)  =f^{\beta}\left(  \lim x_{\sigma}\right)  =\lim
f^{\beta}\left(  x_{\sigma}\right)  =\lim f\left(  x_{\sigma}\right)
=\varphi\left(  f\right)  .
\]
This completes the proof.
\end{proof}

As is well known, a subset $L$ of $C\left(  X\right)  $ is said to
\textsl{separate points from closed sets} if whenever $F$ is a closed set in
$X$ and $x\notin F$, then $f\left(  x\right)  \notin\overline{f\left(
F\right)  }$ for some $f\in L$. Here, $\overline{f\left(  F\right)  }$ denotes
the closure of $f\left(  F\right)  $ in $\mathbb{R}$. Such a subset $L$
determines the topology of $X$, meaning that the topology of $X$ coincides
with the weak topology induced by $L$. Moreover, $X$ turns out to be
completely regular and so a Tychonoff space. In particular, if $L$ separates
points and closed sets in $X$, then a net $\left(  x_{\sigma}\right)
_{\sigma}$ of elements of $X$ converges to some $x\in X$ if and only if, for
every $f\in L$, the net $\left(  f\left(  x_{\sigma}\right)  \right)
_{\sigma}$ converges in $\mathbb{R}$ to $f\left(  x\right)  $. On the other
hand, a truncated vector sublattice $L$ of $C\left(  X\right)  $ is said to be
$L$-\textsl{realcompact} if any truncation homomorphism $\psi$ on $L$ is a
point-evaluation on $L$, that is, there exists $x\in X$ such that%
\[
\psi\left(  f\right)  =f\left(  x\right)  \quad\text{for all }f\in L.
\]
The following result is a consequence of the previous theorem.

\begin{corollary}
\label{realcompact}Let $L$ be a truncated vector sublattice of $C\left(
X\right)  $ which separates points and closed sets. Then the following are equivalent.

\begin{enumerate}
\item[\emph{(i)}] $X$ is $L$-realcompact.

\item[\emph{(ii)}] A net $\left(  x_{\sigma}\right)  $ in $X$ converges in $X$
if and only if the net $\left(  f\left(  x_{\sigma}\right)  \right)  $
converges in $\mathbb{R}$ for every $f\in L$.

\item[\emph{(iii)}] For every $u\in\beta X\backslash X$ there exists $f\in L$
such that $f^{\beta}\left(  u\right)  =\infty$.
\end{enumerate}
\end{corollary}

\begin{proof}
$\mathrm{(i)\Rightarrow(ii)}$ Assume that $X$ is $L$-realcompact and pick a
net $\left(  x_{\sigma}\right)  $ in $X$ such that $\lim f\left(  x_{\sigma
}\right)  $ exists in $\mathbb{R}$ for every $f\in L$. Define $\psi
:L\rightarrow\mathbb{R}$ by putting%
\[
\psi\left(  f\right)  =\lim f\left(  x_{\sigma}\right)  \quad\text{for all
}f\in L.
\]
A short moment's thought reveals that $\psi$ is a truncation homomorphism on
$L$. Hence, there exists $x\in X$ such that%
\[
\psi\left(  f\right)  =f\left(  x\right)  \quad\text{for all }f\in L.
\]
Since $L$ separates points and closed sets, the net $\left(  x_{\sigma
}\right)  $ converges to $x$ in $X$.

$\mathrm{(ii)\Rightarrow(iii)}$ Arguing by contradiction, assume that there is
some $u\in\beta X\backslash X$ for which%
\[
f^{\beta}\left(  u\right)  \in\mathbb{R}\quad\text{for all }f\in L.
\]
By density, there exists a net $\left(  x_{\sigma}\right)  $ in $X$ which
converges to $u$. Hence, if $f\in L$ then%
\[
\lim f\left(  x_{\sigma}\right)  =\lim f^{\beta}\left(  x_{\alpha}\right)
=f^{\beta}\left(  u\right)  \in\mathbb{R}.
\]
In other words, the net $\left(  f\left(  x_{\sigma}\right)  \right)  $
converges in $\mathbb{R}$ and so, by $\mathrm{(ii)}$, the net converges in
$X$. This yields that $u\in X$, a contradiction.

$\mathrm{(iii)\Rightarrow(i)}$ Let $\psi$ be a truncation homomorphism on $L$.
By Theorem \ref{Cech}, there exists $u\in\beta X$ such that%
\[
\psi\left(  f\right)  =f^{\beta}\left(  u\right)  \quad\text{for all }f\in L.
\]
In particular,%
\[
f^{\beta}\left(  u\right)  \in\mathbb{R}\quad\text{for all }f\in L.
\]
By $\mathrm{(iii)}$, the element $u$ must be in $X$, which leads to the conclusion.
\end{proof}

Now we are close to completing the paper, again with a result that gives a
sufficient condition on $L$ for $X$ to be $L$-realcompact. We need first to
recall that if $\mathbb{R}^{I}$ is a product of real lines equipped with its
Tychonoff product topology the, for every $i\in I$, the projection $\pi
_{i}:\mathbb{R}^{I}\rightarrow\mathbb{R}$ defined by%
\[
\pi_{i}\left(  \left(  x_{j}\right)  \right)  =x_{i}\quad\text{for all
}\left(  x_{j}\right)  \in\mathbb{R}^{I}%
\]
is continuous.

\begin{corollary}
Suppose that $X$ is a closed set in an appropriate Tychonoff product space
$\mathbb{R}^{I}$ and let $L$ be a truncated vector sublattice of $C\left(
X\right)  $ such that $\pi_{i}\in L$ for all $i\in I$. The $X$ is $L$-realcompact.
\end{corollary}

\begin{proof}
Since $L$ contains the projections, then $L$ separates points and closed sets.
Hence, we can apply the previous corollary. Let $\left(  x_{\sigma}\right)  $
be a net such that $\left(  f\left(  x_{\sigma}\right)  \right)  $ converges
in $\mathbb{R}$ for all $f\in L$. In particular, for every $i\in I$, the net
$\left(  \pi_{i}\left(  x_{\sigma}\right)  \right)  $ converges in
$\mathbb{R}$, say to $\ell_{i}$. But then $\left(  x_{\sigma}\right)  $
converges in $\mathbb{R}^{I}$ to $x=\left(  \ell_{i}\right)  $. Since $X$ is
closed in $\mathbb{R}^{I}$ we derive that $x\in X$ and $\mathrm{(ii)}$ in
Corollary \ref{realcompact} allows us to conclude.
\end{proof}

\end{document}